\documentclass[12pt]{amsart}

\usepackage{graphicx}
\usepackage{amsmath}
\usepackage{amsfonts}
\usepackage{amssymb}
\usepackage{amsthm}
\usepackage{pstricks}
\usepackage{subfigure}
\usepackage{nicefrac}
\usepackage{setspace}
\newtheorem{theorem}{Theorem}

\newtheorem{lemma}[theorem]{Lemma}

\newtheorem{proposition}[theorem]{Proposition}
\newtheorem{remark}[theorem]{Remark}
\newtheorem*{lem}{Lemma}

\theoremstyle{definition}
\newtheorem{example}[theorem]{Example}

\newenvironment{thmbis}[1]
  {\renewcommand{\thetheorem}{\ref{#1}$'$}%
   \addtocounter{theorem}{-1}%
   \begin{theorem}}
  {\end{theorem}}

\usepackage{color}
\usepackage[normalem]{ulem}

\begin{document}

\title{Solving the social choice problem under equality constraints}

\author[Juan A. Crespo]{Juan A. Crespo}
\address[Juan A. Crespo]{Facultad de Ciencias Econ\'omicas \\ Universidad Aut\'onoma de Madrid \\ 28049 Cantoblanco, Madrid \\ Spain}
\email{juan.crespo@uam.es}

\author[J. J. S\'anchez-Gabites]{J. J. S\'anchez-Gabites}
\address[J. J. S\'anchez-Gabites]{Facultad de Ciencias Econ\'omicas \\ Universidad Aut\'onoma de Madrid \\ 28049 Cantoblanco, Madrid \\ Spain}
\email{JaimeJ.Sanchez@uam.es}

\begin{abstract} In a context where a decision has to be taken collectively by several agents, the social choice problem consists in deciding whether there exists a socially acceptable rule that aggregates the individual preferences of the agents into a social one. We analyze this problem for sets of alternatives defined by equality constraints and obtain a solution that, in sharp contrast to the classical ones, is expressed in an elementary language and ultimately reduces the social choice problem to a standard constrained optimization problem.

By considering a toy example we shall be able to interpret this general result in terms of the rationality (in the sense of Economics) of the design of the set of alternatives rather than, surprisingly, that of the agents involved in the actual social choice problem.
\medskip


\end{abstract}

\maketitle

\section{Introduction}

The \emph{social choice problem} consists in deciding whether there exists a \emph{universal rule} that aggregates the individual preferences of several agents into a social one. More precisely, suppose that an element needs to be selected out of a \emph{set of alternatives} $X$ on the basis of the individual preferences of a number of agents, each of which may well want to choose a different element from $X$. In order to make a collective (or social) decision it is necessary to determine an \emph{ex-ante} rule that, whatever the individual preferences of the agents, is able to aggregate them into a social preference; the element that will finally be selected from $X$. 
The \emph{aggregation rule} is required to satisfy certain axioms which ensure that the aggregation process is performed in a ``socially acceptable'' manner (for instance, it is usually required, at least, that it considers all the agents equally). When no acceptable aggregation rule exists it is customary to speak of a \emph{social choice paradox}.

There are several social choice models in the literature. Since we shall consider sets of alternatives determined by equality constraints, which generally consist of a continuum of points, the most appropriate framework for our purposes seems to be the so-called \emph{topological} social choice model. We will describe it in some detail later on. This model has received considerable attention since it was introduced in the past century, having been independently studied by authors such as Aumann \cite{Aumann1932}, Eckmann \cite{eckmann1}, Hilton \cite{Eck-Hil-Gan1962, Hilton1997}, Chichilnisky and Heal \cite{chiheal1} or, more recently, Weinberger \cite{Weinberger2004} (the interested reader may enjoy Eckmann's historical account \cite{eckmann2}). Their solutions to the social choice problem may be roughly summarized in the following result of Chichilnisky and Heal (\cite[Theorem 1, p. 82]{chiheal1}): 
	
	\smallskip
	{\it Theorem.} Let the space of alternatives $X$ be a parafinite CW complex. The social choice problem over $X$ has an affirmative answer if, and only if, each component of $X$ is contractible.
\smallskip

This result, beautiful as it is, has several drawbacks which may have prevented it from becoming more popular among theorists working in this and related areas:
\begin{itemize}
	\item The language in which it is stated belongs to a rather specialized area of pure mathematics (namely, geometric topology) and is very far from the elementary language of the original social choice problem. In particular, there is no intuitive interpretation of its meaning.
	\item It includes the technical condition that the space of alternatives should be a CW complex, which even in the natural case that we consider in this paper (see below) may not be satisfied at all.
	\item It does not provide a direct solution to the social choice problem but rather translates it into another hard mathematical problem; that of deciding whether a given space is contractible.
\end{itemize}

In this paper we offer a solution to the social choice problem over sets of alternatives $X \subseteq \mathbb{R}^n$ defined by means of a collection of equality constraints\footnote{This class of sets is natural in Economics, but we emphasize that throughout the paper we shall make no assumptions with an economical content or about the particular interpretation (should there be any) of the variables and constraints that define $X$.}; that is, sets of the form \begin{equation} \label{eq:X} X = \{p \in \mathbb{R}^n : g_i(p) = c_i \text{ for } i = 1, \ldots,m\} \end{equation} where the $g_i$ are differentiable maps from $\mathbb{R}^n$ into $\mathbb{R}$. In contrast with the existing results, our solution avoids the issues enumerated above: it is expressed in an elementary language which does not involve any topological concepts, the hypotheses over the set of alternatives are also elementary and easy to check, and the answer is ultimately given in terms of a standard constrained optimization problem.

Although we will take the theorem of Chichilnisky and Heal as our starting point, we would like to remark that sets of the form \eqref{eq:X} may well not be CW complexes, so the theorem does not apply directly to them. Most of the work in this paper is aimed at developing new mathematical techniques to overcome this difficulty.

Our first result provides a necessary condition for the social choice problem over $X$ to have a solution:
\smallskip

{\bf Theorem A.} Consider any one of the constraints $g_i(p) = c_i$ that define $X$. Suppose that the set $Y_i \subseteq \mathbb{R}^n$ defined by the \emph{remaining} constraints is bounded, so that the map $g_i$ attains both a global maximum and minimum value over $Y_i$. Then it is necessary that $c_i$ coincides with either of these for the social choice problem over $X$ to have a solution.
\smallskip

We express the necessary condition of the theorem by saying that the constraint $g_i(p) = c_i$ must be \emph{optimal with respect to the other constraints}, and say that $X$ is \emph{optimally designed} if all the constraints for which $Y_i$ is bounded do indeed satisfy this condition.\footnote{If $Y_i$ is not bounded then Theorem A says nothing about the corresponding $c_i$.} With this language Theorem A can be more expressively stated as follows: \emph{a social choice problem that is not optimally designed will lead to a social choice paradox}. Our choice of the word ``design'' is intended to account for the fact that the values of the $c_i$ have to be carefully selected to avoid a social choice paradox. By contrast, if the constraints $g_i(p) = c_i$ are dictated by some natural or random process (in a nontechnical sense of the word), in all likelihood the social choice problem over $X$ will have no solution. 

It is now natural to ask whether optimal design is not only necessary but also sufficient to avoid social choice paradoxes. Strictly speaking it is not, as shown by Example~\ref{ex:torus} in Section~\ref{sec:sufficient}. However, we will prove that such cases are exceptional, establishing the following:
\smallskip

{\bf Theorem B.} \emph{A generic social choice problem that is designed optimally has a solution.}
\smallskip

Roughly speaking, Theorem B means that for almost every function $g_i$, choosing the value $c_i$ of the constraint to be optimal with respect to the others will define a set $X$ over which the social choice problem can be solved. Admitting that the functions $g_i$ one may encounter in practice fall into this generic category, we may summarize Theorems A and B in the following imprecise but suggestive ``principle of optimal design'':
\smallskip

{\it A problem of social choice under equality constraints has a solution if, and only if, it is optimally designed.}
\smallskip

As announced earlier, the statement of this solution to the social choice problem is elementary and involves no topological notions whatsoever. Also, checking explicitly whether a particular set $X$ is optimally designed is a routine exercise in constrained optimization, so this criterion is readily applicable.

On a more theoretical vein, if one broadly thinks of constrained optimization as one of the concerns of Economics our results show that the latter discipline plays a crucial role in avoiding social choice paradoxes. We now illustrate this with a toy example which leads to an interesting economical interpretation of the principle of optimal design.

Suppose that several agents $A_i$ have to decide collectively on which particular bundle of production factors a firm is going to use. These agents are not concerned with any economical considerations whatsoever, but their choice is at least required to satisfy constraints dictated by the production technology and the external demand. There is no reason to expect that the set of alternatives determined by these constraints will be optimally designed, and so the agents $A_i$ will encounter a social choice paradox (Theorem A).

Now consider another agent $B$ (say the owner of the firm) who is totally independent from the agents $A_i$ and, motivated only by economical considerations, measures the performance of the firm through some utility function $U$ and requires it to attain certain value $u$. This adds one further constraint to the ones mentioned above and reduces the set of alternatives among which the agents $A_i$ can choose. Now:
\begin{itemize}
	\item[(i)] If agent $B$ proceeds in the way that is \emph{economically rational}, she will require the utility level $u$ to be the maximum allowed by the other constraints. Thus, the last constraint will be optimal with respect to the remaining ones and (Theorem B) this will remove the social choice paradox previously encountered by the agents $A_i$.
	\item[(ii)] If agent $B$ does not maximize the utility level but only requires it to attain certain suboptimal value instead, the agents $A_i$ will still encounter a social choice paradox\footnote{Actually, \emph{minimizing} the utility would also avoid a social choice paradox. This alternative is perfectly acceptable from a mathematical point of view but in the case under consideration it can be discarded on economical grounds.} (Theorem A).
\end{itemize}

Let us summarize. The agents $A_i$ face a social choice problem upon whose design they cannot operate and their only goal is to perform their choice in a socially acceptable manner. Agent $B$, who can partially design the set of alternatives but does not take part in the actual social choice (at least not necessarily), is only concerned with the purely economical problem of fixing the utility level $u$ that the firm should achieve. These two problems are, in principle, totally unrelated. The second one is solved axiomatically in Economics by postulating that agent $B$ behaves rationally and therefore maximizes $U$. What we have shown is that by behaving in this way, and only by doing so, agent $B$ simultaneously (and possibly inadvertently) solves also the first problem, avoiding a social choice paradox. Thus the two problems, in spite of their very different nature, have the same solution.\footnote{It might be convenient to remark that this is very different from the situation usually considered in Welfare Economics, where one maximizes a social welfare function which is some sort of average of the utility functions of the individual ones. Here the utility function $U$ of agent $B$ is not assumed to bear any particular relationship to the individual utilities of the agents $A_i$, which in fact play no role in our considerations.}

We finish this Introduction with a brief account of how the rest of the paper is organized. In the following section we recall the basic definitions of the topological social choice model and give the formal statements of Theorems A and B, which are Theorems \ref{teo:main} and \ref{teo:iff}. Their proofs are given in Sections \ref{sec:proof_main} and \ref{sec:sufficient} respectively. Since the arguments are somewhat involved, in both cases we have tried to give first an outline of the main ideas that come into play. Some technical lemmas are postponed to appendices A and B. Some notions from homotopy theory, algebraic topology, differential geometry and Morse theory will be needed to follow the arguments comfortably, especially from Section \ref{sec:sufficient} onwards. We have included suitable references where appropriate.

\section{Formal statement of results} \label{sec:model}

\subsection{The topological social choice model} \label{subsec:framework} As mentioned earlier there are several models of social choice in the literature, differing both in the nature of the input that the agents provide the aggregation rule with and in the axioms that the latter is required to satisfy. For instance, in the well known Arrowian model \cite{Arrow1951} each agent orders all the elements of $X$ and the aggregation rule yields another ordering that, in order to be socially acceptable, is essentially required to be compatible with the orderings established by the agents. However, since in our context $X$ will in general consist of a continuum of alternatives, it does not seem reasonable to require that the agents order all of them, but simply choose their favourite one $p_i$ instead. An aggregation rule will therefore be a mapping $F(p_1,p_2,\ldots,p_k) = p$, where the $p_i$ are the choices of the $k$ agents and $p$ is the social choice. Since the agents have absolute freedom in choosing their preferred alternative, $F$ should be defined for any tuple $(p_1,p_2,\ldots,p_k) \in X \times \stackrel{(k)}{\ldots} \times X$; that is, it should be a mapping $F : X \times \stackrel{(k)}{\ldots} \times X \longrightarrow X$. Moreover, it is natural to assume that the agents cannot distinguish between two alternatives that are extremely close to each other and, as a consequence, switching from one to the other should only cause a small change in the aggregate choice $F(p_1,\ldots,p_k)$. This is ensured if $F$ is continuous, and it is this property what is characteristic of the topological social choice model introduced by Chichilnisky \cite{chichil1, chichil2} (for more details on this see \cite[after Remark 2.3.1, p. 6]{lauwers2000}). Finally, for the aggregation rule $F$ to be socially acceptable, it is required to satisfy the following two properties: (i) all the agents should be equally considered (\emph{anonimity}), (ii) if all the agents happen to choose the same element $p \in X$, then the aggregated choice should be that same element $p$ (\emph{unanimity}).

Summarizing, a \emph{social choice function} is a continuous mapping $F : X \times \stackrel{(k)}{\ldots} \times X \longrightarrow X$ that satisfies the following axioms:
\begin{itemize}
	\item[(A)] $F(p_1,p_2,\ldots,p_k)$ should be independent of the ordering of the $p_i$ (anonimity).
	\item[(U)] $F(p,p,\ldots,p)=p$ for every $p \in X$ (unanimity).
\end{itemize}

(Notice that it makes sense to speak of continuity since $X$ is a subset of some $\mathbb{R}^n$). An aggregation rule $F$ that satisfies the above three axioms is called a \emph{social choice function} over the given set of alternatives $X$. With this terminology, the social choice problem can be stated as follows: \emph{given a set of alternatives $X$, does there exist a social choice function over $X$}?

When $X$ is simple enough the existence of social choice functions can sometimes be easily established directly. For example, when $X$ is a convex subset of $\mathbb{R}^n$ then the mean $F(p_1,\ldots,p_k) := \frac{1}{k}\sum_{i=1}^k p_i$ is a social choice function over $X$. Of course, it can no longer be assured to be a social choice function over a set $X$ defined by equality constraints because as soon as any of the defining constraints is nonlinear the mean of two elements from $X$ does not need to belong to $X$ again.\footnote{In passing, let us observe that if all the constraints defining $X$ are linear then $X$ is indeed a convex subset of $\mathbb{R}^n$ and therefore the mean solves the social choice problem over $X$. This does not contradict our results because none of the sets $Y_i$ obtained by removing one of the constraints is bounded.}

By contrast, showing that a given set of alternatives $X$ does not admit any social choice function is a quite challenging mathematical problem even when $X$ is a very simple set such as the unit circumference in the plane $\mathbb{R}^2$ (a case analyzed by Chichilnisky \cite{chichil1} when considering linear preferences on the commodity space of bundles of two collective goods). A paper by Baigent \cite{Baigent2011} gives a rather complete idea of the argument starting from a minimal mathematical background.



\subsection{Our results} \label{subsec:statements} As mentioned earlier, we are interested in the social choice problem over sets of alternatives $X \subseteq \mathbb{R}^n$ that are defined by a collection of equality constraints. Thus, let there be a collection of smooth maps $g_1, \ldots, g_m : \mathbb{R}^n \longrightarrow \mathbb{R}$ and values $c_1, \ldots, c_m \in \mathbb{R}$ which determine the set of alternatives $X$ as \[X = \{p \in \mathbb{R}^n : g_i(p) = c_i \text{ for every } 1 \leq i \leq m\}.\]

The constraints can be completely arbitrary (in particular, they do not need to be linear) but we shall always assume that the set of alternatives $X$ they determine is bounded, which is a reasonable requisite in most problems that try to capture some aspect of reality. Of course, for the social choice problem to make sense $X$ should be nonempty. Also, if $X$ is finite then the problem has a somewhat trivial answer in the affirmative (see Section \ref{sec:sufficient}), so the case of interest is when $X$ is actually infinite. A convenient way of encapsulating these considerations consists in requiring that the number of constraints $m$ is strictly smaller than the dimension of the ambient space $n$; that is, $m < n$.
\medskip

We will first consider an auxiliary base case assuming that the $g_i$ satisfy the standard \emph{constraint qualifications}; that is, the gradients of the $g_i$ are linearly independent at each $p \in X$. Then we will definitely run into a social choice paradox:

\begin{proposition}\label{prop:production} Let the set of alternatives $X$ be bounded, $m < n$, and assume that the $g_i$ satisfy the constraint qualifications. Then the social choice problem over $X$ has no solution.
\end{proposition}

Although the proof of Proposition \ref{prop:production} will be rather easy given the appropriate tools, it is a quite useful result. For example, the classical case when $X$ is a sphere can be analyzed very easily: an $(n-1)$--dimensional sphere is described by a single ($m=1$) implicit equation $x_1^2 + \ldots + x_n^2 = 1$ in $\mathbb{R}^n$ which evidently satisfies the constraint qualifications, so the social choice problem over spheres has no solution for $n \geq 2$.

Sard's theorem from differential geometry (see for instance \cite[p. 227]{kosinski1}) states that the set of vectors $(c_1,\ldots,c_m) \in \mathbb{R}^m$ for which the constraints $g_i(p) = c_i$ do not satisfy the constraint qualifications has Lebesgue measure zero. Otherwise stated, with probability one the values $c_i$ that appear at the right hand side of the constraints $g_i(p) = c_i$ are such that the $g_i$ satisfy the constraint qualifications. It then follows from Proposition \ref{prop:production} that a generic social choice problem under equality constraints will have no solution. This accords with the claim of Theorem A that only problems that are designed in a very specific way (that is, optimally) can be solved.
\medskip

Let us move on to the precise statement of Theorem A. Observe that Proposition \ref{prop:production} entails that, in order to have any hope of solving the social choice problem over $X$, it is necessary that the $g_i$ do not satisfy the constraint qualifications. It will be enough to consider the simplest case, when $m-1$ of the constraints do satisfy them and it is only the addition of the remaining one that spoils this condition. For definiteness we shall assume that it is the first $m-1$ constraints that satisfy the constraint qualifications; that is, the gradients of $g_1, \ldots, g_{m-1}$ are linearly independent at each point of the feasible set they determine \[Y_m = \{p \in \mathbb{R}^n : g_1(p) = c_1, \ldots, g_{m-1}(p) = c_{m-1}\}.\]

Having chosen the last constraint as the one on which we are going to focus, in the sequel we shall safely omit the subindex from $Y_m$ and simply write $Y$ without risk of confusion. Then the precise statement of Theorem A is as follows:

\begin{theorem} \label{teo:main} Let the set $Y$ be bounded and connected, $m < n$, and assume that the first $m-1$ constraints satisfy the constraint qualifications.\footnote{Also, as a technical hypothesis, it is necessary to assume that $g_m|_Y$ has only finitely many critical values. This condition will be fulfilled in any problem with a reasonable economical interpretation, so in practice it is barely stringent. For instance, it is automatically satisfied whenever $g_m$ is an analytic function (polynomials being the simplest case).} If the social choice problem over $X$ has a solution then the last constraint must be optimal with respect to the remaining ones; that is, $c_m$ must be a global optimum of $g_m|_Y$.
\end{theorem}

The connectedness assumption on $Y$ means that it consists only of a single ``piece'' and is included just for convenience: if $Y$ is not connected, that is, if it consists of several disjoint pieces, then the conclusion of the theorem is that $c_m$ must be the global optimum value of $g_m$ restricted to one of those pieces.

Now suppose that the first $m-1$ constraints that define $Y$ are already given (and, as before, satisfy the constraint qualifications) and a designer, aware of the necessary condition given by Theorem \ref{teo:main} and willing to avoid a social choice paradox, sets the value $c_m$ of the last constraint to be optimal with respect to the others. Denote by $\mathcal{C}^{\infty}(\mathbb{R}^n,\mathbb{R})$ the set of all differentiable mappings $g : \mathbb{R}^n \longrightarrow \mathbb{R}$. Then the following result, which is the formal counterpart of Theorem B from the Introduction, holds:

\begin{theorem} \label{teo:iff} Let the set $Y$ be bounded and connected, $m < n$, and assume that the first $m-1$ constraints satisfy the constraint qualifications. There is an open and dense set $\mathcal{M} \subseteq \mathcal{C}^{\infty}(\mathbb{R}^n,\mathbb{R})$ such that whenever $g_m$ belongs to $\mathcal{M}$, setting the last constraint $g_m(p) = c_m$ to be optimal with respect to the others ensures that the social choice problem over $X$ has a solution in the affirmative.
\end{theorem}

The openness and density of $\mathcal{M}$ can be interpreted intuitively as meaning that functions in $\mathcal{M}$ are \emph{robust} (a small perturbation of a function in $\mathcal{M}$ still belongs to $\mathcal{M}$) and very \emph{abundant} (any function can be turned into a function in $\mathcal{M}$ by a perturbation as small as desired). It is in this sense that the word ``generic'', common in Functional Analysis, was used in the statement of Theorem B.

\section{The necessity of optimal design (proof of Theorem \ref{teo:main})} \label{sec:proof_main}

Since the proof of Theorem \ref{teo:main} is somewhat involved we have thought it convenient to begin with an outline of the argument. This is the content of the first subsection, where we also prove Proposition \ref{prop:production}. After developing some auxiliary results in \ref{subsec:hscf} and \ref{subsec:extend}, the proof of the theorem is finally given in subsection \ref{subsec:proof}.

We will need the following two lemmas concerning the contractibility of manifolds:

\begin{lemma} \label{lem:var1} Let $M$ be a compact manifold of dimension $d \geq 1$ and without boundary. Then none of the components of $M$ is contractible.
\end{lemma}

\begin{lemma} \label{lem:var2} Let $M$ be a compact contractible manifold of dimension $d \geq 2$. Then its boundary $\partial M$ is nonempty (by the previous lemma) and connected.
\end{lemma}

In proving these lemmas it seems unavoidable to make use of homology with real coefficients, which is a powerful tool from algebraic topology. Since this machinery might not be familiar to the reader, we have postponed the proof of both lemmas to Appendix A.

\subsection{Outline of the argument} We will consider the whole family of social choice problems that arise as $c_m$ runs in the real numbers, thus changing the set of alternatives $X$. To emphasize that $c_m$ now plays the role of a parameter we shall replace it by $u$ and reflect this explicitly in the notation for $X$, letting \[X_u = \{p \in \mathbb{R}^n : g_i(p) = c_i \text{ for $1 \leq i \leq m-1$ and } g_m(p) = u\}.\] This can be equivalently described as \[X_u = \{p \in Y : g_m(p) = u\},\] where $Y$ is the set defined by the first $(m-1)$ constraints as introduced earlier. If we denote by $u_{\rm min}$ and $u_{\rm max}$ the global minimum and maximum values of $g_m$ over $Y$, Theorem \ref{teo:main} amounts to the following statement:
\smallskip

(*) \emph{If $u \in (u_{\rm min},u_{\rm max})$ then there is no social choice function over $X_u$.}
\smallskip

We shall articulate the proof of (*) in two cases according to whether $u$ is a regular or a critical value of $g_m|_Y$. Recall that a point $p \in Y$ is called a \emph{critical point} of the restricted map $g_m|_{Y}$ if the gradient $\nabla g_m(p)$ is a linear combination of $\{\nabla g_1(p),\ldots,\nabla g_{m-1}(p)\}$, and in that case $u = g_m(p)$ is said to be a \emph{critical value} of $g_m|_Y$. A value $u$ that is not critical is called \emph{regular}.

Consider first the case when $u$ is a regular value of $g_m|_Y$. This amounts to saying that the gradients of $g_1, \ldots, g_m$ are all linearly independent at each $p \in X_u$; that is, the constraint qualifications are satisfied. This is precisely the situation considered in Proposition \ref{prop:production}. Geometrically, it implies that (if nonempty) $X_u$ is a differentiable manifold of dimension $d = n - m \geq 1$ without boundary (see for instance \cite[Theorem 2.3, p. 213]{delafuente1}). Now, any differentiable manifold $M$ (with or without boundary) is indeed a parafinite CW complex because it can be triangulated, as shown by Whitehead \cite{whitehead1940} or Whitney \cite[Theorem 12A, p. 124]{whitney1957}. Thus the theorem of Chichilnisky and Heal directly reduces the problem to deciding whether the components of $X_u$ are contractible. The assumption that $X_u$ is bounded, together with the fact that $X_u$ is closed in $\mathbb{R}^n$, entails that $X_u$ is compact. Hence we can apply Lemma \ref{lem:var1} to $M = X_u$ to learn that none of the components of $X_u$ is contractible and conclude that the social choice problem over $X_u$ has no solution. This proves Proposition \ref{prop:production}.

The above establishes (*) when $u$ is a regular value of $g_m|_Y$, so we are left to consider the case when $u$ is a critical value of $g_m|_Y$; that is, we need to prove
\smallskip

(**) \emph{If $u \in (u_{\rm min},u_{\rm max})$ is a critical value of $g_m|_Y$, then there is no social choice function over $X_u$.}
\smallskip

This requires some hard work. The reason is that when $u$ is a critical value of $g_m|_Y$ there is no guarantee that $X_u$ is a manifold or even a CW complex anymore, and so we are not entitled to apply neither the theorem of Chichilnisky and Heal nor Lemma \ref{lem:var1}, as we did before, to reach a conclusion. To illustrate this point one may consider the situation depicted in Figure \ref{fig:explanation}.(a), where $X \subseteq \mathbb{R}^3$ is defined by two constraints. The first constraint alone defines a surface $Y$. Let the other constraint be given by the map $g_2(x,y,z) = z$. The $Z$ axis is represented vertically, so the various sets $X_u$ are simply the intersections of $Y$ with horizontal planes at height $u$. The critical points of $g_2|_Y$ are $p$ and $q$ (and possibly others not shown in the picture) and $u$ is a critical value of $g_2|_Y$ because $X_u$ contains the critical point $p$. Observe that $X_u$ is an ``eight-figure'' (two circumferences having a single point in common), which is not a manifold as mentioned earlier.\footnote{In this simple drawing $X_u$ is still a CW--complex, but in more complicated situations this need not be the case.}

\begin{figure}[h!]
\null\hfill
\subfigure[]{
\begin{pspicture}(0,0)(8,4.5)
\rput[bl](0.5,0){\scalebox{0.5}{\includegraphics{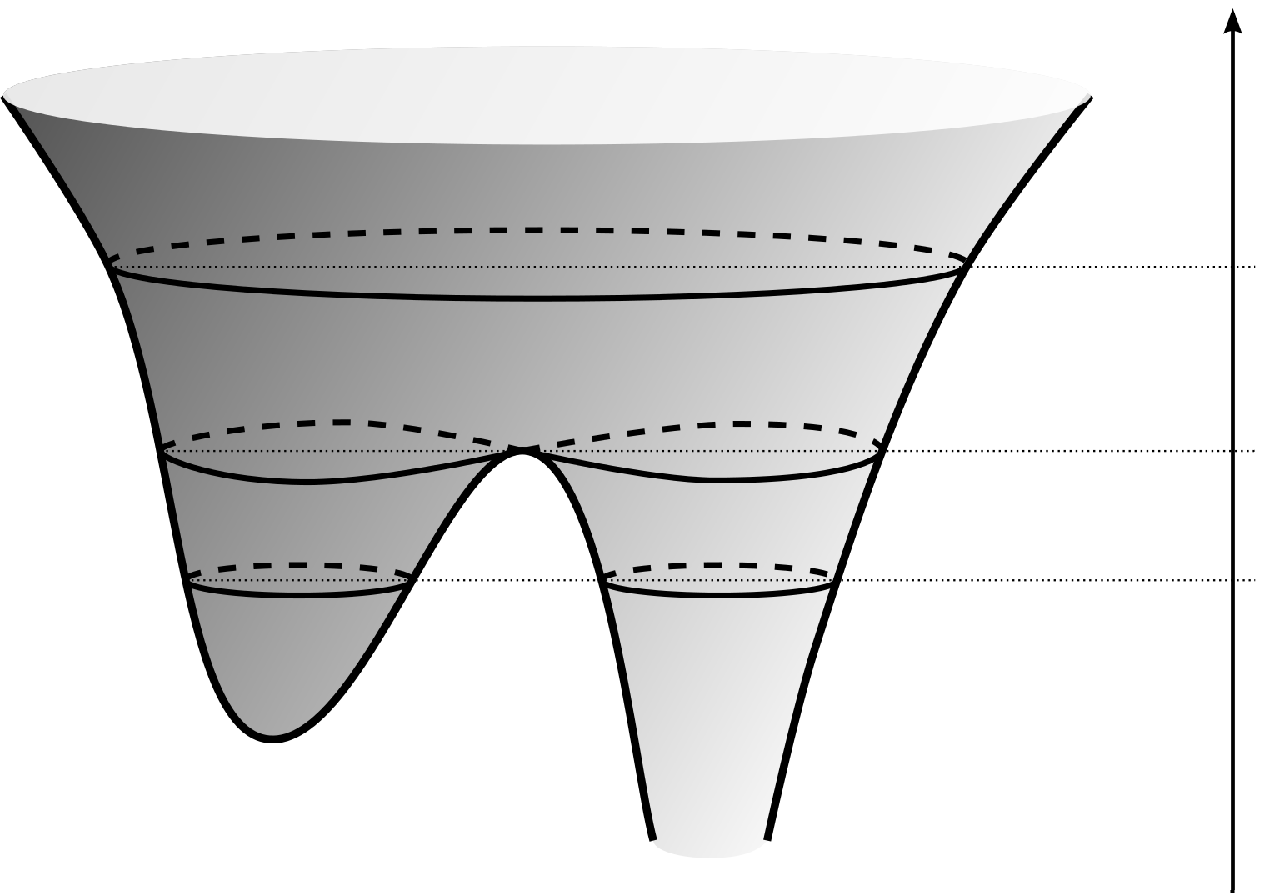}}}
\rput(2,4.2){$Y$} \rput(3.15,2){$p$} \rput(1.85,0.6){$q$}
\rput(7.2,1.6){$u_1$} \rput(7.2,2.3){$u$} \rput(7.2,3.2){$u_2$}
\end{pspicture}}
\hfill
\subfigure[]{
\begin{pspicture}(-0.6,0)(6.8,4.5)
\rput[bl](0.5,0){\scalebox{0.5}{\includegraphics{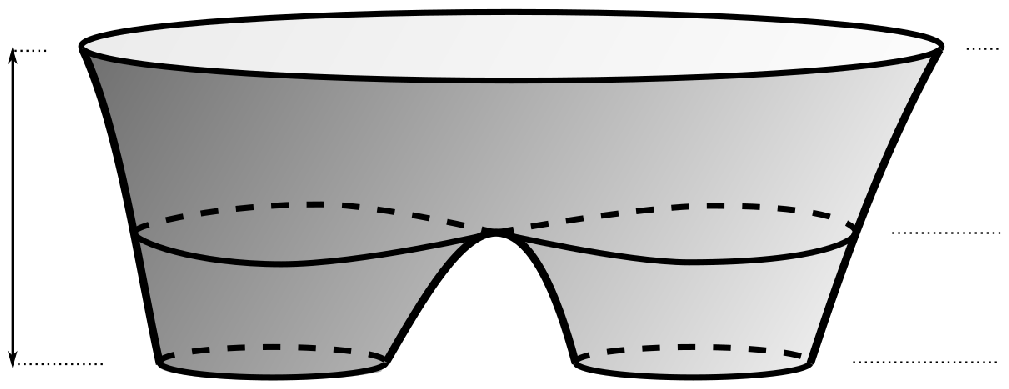}}}
\rput[l](5.8,1.4){$X_{u_1}$} \rput[l](5.8,2.1){$X_u$} \rput[l](5.8,3){$X_{u_2}$}
\rput(0,2.2){$X_{[u_1,u_2]}$}
\end{pspicture}}
\hfill\null
\caption{\label{fig:explanation}}
\end{figure}

To establish (**) we will argue by contradiction. Suppose that there does exist a social choice function over $X_u$ where, we recall, $u_{\rm min} < u < u_{\rm max}$ and $u$ is a critical value of $g_m|_Y$. Pick two numbers $u_1$ and $u_2$ such that $u_{\rm min} < u_1 < u < u_2 < u_{\rm max}$ and consider the auxiliary set \[X_{[u_1,u_2]} = \{p \in Y : u_1 \leq g_m(p) \leq u_2\}.\] One can always choose $u_1$ and $u_2$ to be regular values of $g_m|_Y$, and this guarantees that the set $M = X_{[u_1,u_2]}$ is a compact manifold whose boundary $\partial M$ is the disjoint union of $X_{u_1}$ and $X_{u_2}$, where as usual \[X_{u_1} = \{p \in Y : g_m(p) = u_1\} \quad \text{and} \quad X_{u_2} = \{p \in Y : g_m(p) = u_2\}.\] In our illustrative example of Figure \ref{fig:explanation} the set $X_{[u_1,u_2]}$ is the region of $Y$ comprised between heights $u_1$ and $u_2$, as shown in panel (b), and $X_{u_1}$ and $X_{u_2}$ are the bottom and top ``rims'' of $X_{[u_1,u_2]}$, respectively. We will prove that:
\begin{itemize}
	\item[(i)] \label{extend} With an appropriate choice of $u_1$ and $u_2$ the social choice function that exists over $X_u$ by assumption can ``almost'' be extended to another one defined on all $X_{[u_1,u_2]}$.
  \item[(ii)] As a consequence of (i) and the theorem of Chichilnisky and Heal applied to the ma{\-}ni{\-}fold $M = X_{[u_1,u_2]}$, it follows that the latter must be contractible and therefore its boundary $\partial M$ must be connected by Lemma \ref{lem:var2}. However, since $\partial M$ is the disjoint union of $X_{u_1}$ and $X_{u_2}$, at least one of them must be empty.
\end{itemize}

We reach a contradiction from (ii) as follows. Since $Y$ is compact and connected, $g_m(Y) \subseteq \mathbb{R}$ is also compact and connected, so it is a compact interval. Therefore $g_m(Y) = [u_{\rm min},u_{\rm max}]$. In particular both $u_1$ and $u_2$ belong to $g_m(Y)$, and this implies that both $X_{u_1}$ and $X_{u_2}$ are nonempty.  This contradicts (ii), showing that the assumption that there is a social choice function over $X_u$ is untenable and completing the proof of Theorem \ref{teo:main}.

Both (i) and (ii) are, in fact, more delicate than they seem. First, it may not always be possible to extend a social choice function originally defined only on $X_u$ to a social choice function defined on the larger set $X_{[u_1,u_2]}$. We shall see, however, that it is always possible to extend it at the homotopy level (hence the use of the word ``almost'' in step (i) above). This will prompt the introduction of \emph{homotopy social choice functions} below, a notion which is slightly weaker than that of a true social choice function but still good enough for our purposes. Secondly, in step (ii) the application of the theorem of Chichilnisky and Heal only shows that each component of $X_{[u_1,u_2]}$ must be contractible, but maybe not $X_{[u_1,u_2]}$ as a whole. Thus we can only conclude that each component of $X_{[u_1,u_2]}$ has a connected boundary, which does not directly lead to a contradiction and therefore will require a more detailed analysis of the situation.

As mentioned in the statement of Theorem \ref{teo:main}, we will need to assume that $g_m|_Y$ has at most finitely many critical values (notice that since many critical points may correspond to the same critical value, $g_m|_Y$ may well have infinitely many critical points in spite of having only finitely many critical values as required).
\smallskip

{\it Note.} For the sake of brevity, from now on we shall sometimes abbreviate ``social choice function'' as SCF.


\subsection{Homotopy social choice functions} \label{subsec:hscf} Let us recast the conditions of unanimity and anonimity in a slightly different but well known equivalent way. Suppose $F$ is an SCF for $k$ agents on a space of alternatives $X$. Denote by the letter $\Delta$ the diagonal map \[\Delta : X \longrightarrow X^k \ \ ; \ \ \Delta(p) = (p,p,\ldots,p)\] and by the letter $P$ any permutation map $P : X^k \longrightarrow X^k$. The conditions of unanimity (U) and anonimity (A) on $F$ can then be equivalently stated as
\begin{enumerate}
        \item[(U)] $F \circ \Delta = {\rm id}$,
        \item[(A)] $F \circ P = F$ for any permutation map $P$.
\end{enumerate}

As it turns out, the proof of the theorem of Chichilnisky and Heal works equally well if (U) and (A) only hold at the homotopy level; that is, if they are replaced by
\begin{enumerate}
        \item[(HU)] $F \circ \Delta \simeq {\rm id}$,
        \item[(HA)] $F \circ P \simeq F$ for any permutation map $P$,
\end{enumerate}
where $\simeq$ denotes the homotopy relation between maps. The reason is that, when one considers the maps $F^*$, $\Delta^*$ and $P^*$ induced by $F$, $\Delta$ and $P$ between homotopy groups, unanimity (U) and its homotopical counterpart (HU) yield the same relation $F^* \circ \Delta^* = {\rm id}$, and the same goes for (A) and (HA) (namely, $F^{*} \circ P^* = F^*$). Since it is only these relations that are needed to conclude that the components of $X$ are contractible, our assertion follows.

For the sake of brevity let us call a continuous map $F : X^k \longrightarrow X$ a \emph{homotopy SCF} (for $k$ agents) if it satisfies conditions (HU) and (HA) above. These functions are not to be interpreted in any intuitive sense, but just as mathematical objects that will be useful to prove Theorem \ref{teo:main}. Our discussion may be summed up in the following

\begin{lemma} \label{lem:chichil} Let $M$ be a compact differentiable manifold. Assume there exist homotopy social choice functions over $M$ for any number of agents. Then each component of $M$ is contractible.
\end{lemma}

\subsection{An extension result} \label{subsec:extend} Now we are going to establish the extension result mentioned in (i) in page \pageref{extend}; namely, that an SCF on $X_u$ can be extended to a \emph{homotopy} SCF on $X_{[u_1,u_2]}$.

Let $u$ be a critical value of $g_m|_Y$ and suppose that it is not a global optimum, so that $u_{\rm min} < u < u_{\rm max}$. As a consequence of the technical assumption that $g_m|_Y$ only has finitely many critical values we may choose $u_1,u_2 \in [u_{\rm min},u_{\rm max}]$ such that $u_1 < u < u_2$ and $u$ is the only critical value of $g_m|_Y$ on the interval $[u_1,u_2]$.

\begin{proposition} \label{prop:extend} Let $u_1$ and $u_2$ be chosen as above. If there exists a social choice function $F$ for $k$ agents on $X_u$, then there exists a homotopy social choice function $F'$ for $k$ agents on $X_{[u_1,u_2]}$.
\end{proposition}

The proof of the proposition needs Lemmas \ref{lem:extend1} and \ref{lem:morse}, which we state now. Their proofs are postponed to Appendix B, since they are slightly technical.

\begin{lemma} \label{lem:extend1} There exist a neighbourhood $U$ of $X_u$ in $X_{[u_1,u_2]}$ and a continuous function $F_U : U^k \longrightarrow X_{[u_1,u_2]}$ with the properties
\begin{enumerate}
        \item $F_U(p,\ldots,p) = p$,
        \item $F_U(p_1,\ldots,p_k)$ is independent of the ordering of the $p_i$.
\end{enumerate}
\end{lemma}

Notice that $F_U$ is close to being an SCF on $U$ (it is certainly unanimous and anonymous), but it does not qualify as such because its target space is $X_{[u_1,u_2]}$ rather than $U$. To remedy this we shall make use of the following lemma:

\begin{lemma} \label{lem:morse} Given any neighbourhood $U$ of $X_u$ in $X_{[u_1,u_2]}$ there exists a continuous mapping $r : X_{[u_1,u_2]} \longrightarrow X_{[u_1,u_2]}$ such that:
\begin{enumerate}
        \item[(1)] $r(p) \in U$ for every $p \in X_{[u_1,u_2]}$,
        \item[(2)] $r$ is homotopic to the identity in $X_{[u_1,u_2]}$.
\end{enumerate}
\end{lemma}

Now we put these two results together to prove Proposition \ref{prop:extend}.

\begin{proof}[Proof of Proposition \ref{prop:extend}] Apply Lemma \ref{lem:extend1} to find $U$ and $F_U$; then apply Lemma \ref{lem:morse} to the $U$ just obtained to get $r$. Define $F' : X_{[u_1,u_2]}^k \longrightarrow X_{[u_1,u_2]}$ by \[F'(p_1,\ldots,p_k) := F_U(r(p_1),\ldots,r(p_k)).\] Notice that the definition is correct: all the $r(p_i)$ belong to $U$ and therefore it makes sense to evaluate $F_U$ on the $k$--tuple $(r(p_1),\ldots,r(p_k))$.

We claim that $F'$ is a homotopy SCF on $X_{[u_1,u_2]}$. Clearly $F'$ is insensitive to the ordering of its arguments because the same is true of $F_U$, so (HA) holds. Also, composing $F'$ with the diagonal map $\Delta(p) = (p,\ldots,p)$ yields \[F' \circ \Delta(p) = F_U(r(p),\ldots,r(p)) = r(p),\] and since $r \simeq {\rm id}$ in $X_{[u_1,u_2]}$, we see that \[F' \circ \Delta \simeq {\rm id}.\] This establishes property (HU) and shows that $F'$ is indeed a homotopy SCF.
\end{proof}

\subsection{The proof of Theorem \ref{teo:main}} \label{subsec:proof} We only need to prove (**). To argue by contradiction assume that there exists an SCF over $X_u$ where $u_{\rm min} < u < u_{\rm max}$ is a critical value of $g_m|_Y$. Choose $u_1$ and $u_2$ satisfying $u_{\rm min} < u_1 < u < u_2 < u_{\rm max}$ and such that $u$ is the only critical value of $g_m|_Y$ in the interval $[u_1,u_2]$ (this is possible because of the technical assumption that $g_m|_Y$ only has finitely many critical values). The set $X_{[u_1,u_2]}$ is then a manifold because both $u_1$ and $u_2$ are regular values of $g_m|_Y$. It is clearly compact and its dimension is the same as that of $Y$, which is $d = n - (m-1) \geq 2$ because of the condition $m < n$. The boundary of $X_{[u_1,u_2]}$ is the disjoint union of $X_{u_1}$ and $X_{u_2}$. We proved earlier, using the connectedness and compactness of $Y$, that both $X_{u_1}$ and $X_{u_2}$ are nonempty.

By Proposition \ref{prop:extend} the SCF that exists over $X_u$ by assumption can be extended to a homotopy SCF on $X_{[u_1,u_2]}$. As a consequence of Lemma \ref{lem:chichil} it follows that each component $C$ of $X_{[u_1,u_2]}$ is contractible, and so its boundary $\partial C$ is connected by Lemma \ref{lem:var2}. In particular, for each $C$ we have that $\partial C$ must be \emph{entirely} contained in either $X_{u_1}$ or $X_{u_2}$. Thus, letting \[P := \cup \{C : C \text{ connected component of } X_{[u_1,u_2]} \text{ such that } \partial C \subseteq X_{u_1}\}\] and \[Q := \cup \{C : C \text{ connected component of } X_{[u_1,u_2]} \text{ such that } \partial C \subseteq X_{u_2}\}\] we have that $X_{[u_1,u_2]}$ is the disjoint union of $P$ and $Q$. Notice that both $P$ and $Q$ are nonempty because both $X_{u_1}$ and $X_{u_2}$ are nonempty too.

Pick points $p \in P$ and $q \in Q$. Since $Y$ is a connected (by assumption) manifold, it is also path connected. Thus there exists a continuous path $\gamma : [0,1] \longrightarrow Y$ such that $\gamma(0) = p$ and $\gamma(1) = q$. Since $\gamma$ starts in $P$ and ends in $Q$ and $P$ and $Q$ are disjoint, there exist a number $t_P$ when $\gamma$ last belongs to $P$ and a number $t_Q > t_P$ when $\gamma$ first hits $Q$. Formally these are \[t_P := \sup\ \{ t \in [0,1] : \gamma(t) \in P\} \quad \text{and} \quad t_Q := \inf\ \{ t \in [t_P,1] : \gamma(t) \in Q\}.\] Observe the following:
\begin{itemize}
	\item[(a)] By definition, $\gamma(t) \not\in X_{[u_1,u_2]}$ for every $t \in (t_P,t_Q)$.
	\item[(b)] By continuity $\gamma$ must exit $P$ through its boundary and similarly enter $Q$ through its boundary, so $\gamma(t_P) \in \partial P$ and $\gamma(t_Q) \in \partial Q$.
\end{itemize}

The definition of $P$ and $Q$ together with (b) above imply that $g_m(\gamma(t_P)) = u_1$ and $g_m(\gamma(t_Q)) = u_2$. Consider the set $g_m(\gamma([t_P,t_Q]))$. It is a connected, compact subset of $\mathbb{R}$, so it must be a compact interval. Since it contains $g_m(\gamma(t_P)) = u_1$ and $g_m(\gamma(t_Q)) = u_2$, it must also contain $u$. Thus there exists $t_* \in [t_P,t_Q]$ such that $g_m(\gamma(t_*)) = u$, which means that $\gamma(t_*)$ belongs to $X_u$ and therefore also to $X_{[u_1,u_2]}$. This contradicts (a) above and finishes the proof of Theorem \ref{teo:main}.

\section{The sufficiency of optimal design (Theorem \ref{teo:iff})} \label{sec:sufficient}

We begin by presenting a trivial example that illustrates how the optimality condition of Theorem \ref{teo:main} is generally not sufficient to avoid a social choice paradox:

\begin{example} \label{ex:torus} Consider the set of alternatives \[X = \{(x,y,z) \in \mathbb{R}^3 : (x^2+y^2+z^2+3)^2 - 16(x^2+y^2) = 0 \ , \ z = 1\}.\]

It is easy to check that the second constraint is optimal with respect to the first one, so the necessary condition provided by Theorem \ref{teo:main}. (Notice that the set $Y_2$ defined by the second constraint alone is not bounded, so Theorem \ref{teo:main} does not require the first constraint to be optimal with respect to the second one.) A straightforward computation shows that $X$ is a circumference of radius $2$ centered at $(0,0,1)$ and contained in the plane $z = 1$. However, circumferences do not admit social choice functions. ${}_{\blacksquare}$
\end{example}

Let us explain now more carefully the statement of Theorem \ref{teo:iff}. Consider once more the bounded set $Y$ defined by the first $m-1$ constraints alone. These should be thought of as being fixed once and for all, and we imagine that the last constraint $g_m(p) = c_m$ is a parameter so that the map $g_m$ and the number $c_m$ can vary, yielding a whole family of sets of alternatives $X_{g_m}^{c_m}$. However, since we are only interested in the case when the last constraint is optimal with respect to the others, for each map $g_m$ there are only two possible choices of $c_m$: either the global maximum $c_m^{\rm max}$ or the global minimum $c_m^{\rm min}$ of $g_m|_Y$. The content of Theorem \ref{teo:iff} is that for most choices of $g_m$ both possibilities lead to a set of alternatives where the social choice problem has a solution.

Let us formalize this idea. For any smooth map $g_m : \mathbb{R}^n \longrightarrow \mathbb{R}$ denote $c_m^{\rm max} = \max\ g_m|_Y$ and $c_m^{\rm min} = \min\ g_m|_Y$ (since $Y$ is bounded by assumption, these two numbers are well defined) and consider the two sets of alternatives \[X_{g_m}^{\rm max} = \{p \in \mathbb{R}^n : g_1(p) = c_1, \ldots, g_{m-1}(p) = c_{m-1}, g_m(p) = c_m^{\rm max}\}\] and \[X_{g_m}^{\rm min} = \{p \in \mathbb{R}^n : g_1(p) = c_1, \ldots, g_{m-1}(p) = c_{m-1}, g_m(p) = c_m^{\rm min}\}.\]

With this notation, Theorem \ref{teo:iff} can be stated in more detail as follows:

\begin{thmbis}{teo:iff} \label{teobis:iff} Let the set $Y$ be bounded and connected, $m < n$, and assume that the first $m-1$ constraints satisfy the constraint qualifications. There is a set $\mathcal{M} \subseteq \mathcal{C}^{\infty}(\mathbb{R}^n,\mathbb{R})$ which is open and dense (with the strong $\mathcal{C}^2$--topology) and such that, when $g_m$ belongs to this set, the social choice problems over both $X_{g_m}^{\rm max}$ and $X_{g_m}^{\rm min}$ admit a solution.
\end{thmbis}

The intuitive interpretation of the openness and density properties of $\mathcal{M}$ was mentioned in Section \ref{sec:model} and captures the idea of ``genericity''. The strong $\mathcal{C}^2$--topology can be most easily described by saying that two functions $g_1,g_2$ are $\epsilon$--close when the functions themselves, together with their derivatives up to second order, differ by no more than $\epsilon$ at each point of $\mathbb{R}^n$.

\subsection{Outline of the argument} The basic idea behind Theorem \ref{teo:iff} is that, when the last constraint is optimal with respect to the others and $g_m \in \mathcal{M}$, the set of alternatives $X$ actually reduces to a \emph{finite} number of points, and in this case there do exist social choice functions over $X$. To see why this last assertion is true, begin by labelling the finitely many alternatives (that is, the elements of $X$) in any order. Then, given the bundle of individual preferences $(p_1,\ldots,p_k)$, simply let $F(p_1,\ldots,p_k)$ be that alternative, among those that appear in $(p_1,\ldots,p_k)$, having the highest label. It is easy to check that $F$ is anonymous, unanimous and (trivially) continuous, so it is a social choice function over $X$. (It is, however, questionable to what extent such a function is actually of interest in the realm of social choice.)

We now need to recall a couple of definitions pertaining to Morse theory. Let $M$ be a compact, boundariless manifold, and let $h : M \longrightarrow \mathbb{R}$ be a smooth function (eventually we will take $M = Y$ and $h = g_m|_Y$). A critical point $p \in M$ for $h$ is said to be \emph{nondegenerate} if the matrix of second partial derivatives of $h$ at $p$ is nonsingular; that is, it has a nonzero determinant\footnote{One would express $h$ in local coordinates around $p$ and construct the matrix of second partial derivatives of this local expression. Whether or not this matrix has a nonzero determinant turns out to be independent of the coordinates used to compute it, and this makes the above definition valid. The interested reader can find more information about this in \cite[Chapter 2, pp. 33ff.]{matsumoto1}}. The map $h$ itself is called a \emph{Morse function} if all its critical points are nondegenerate.

By their very definition, nondegenerate critical points of a map $h$ have the property that the quadratic term in the Taylor expansion of $h$ around them is either definite or indefinite, but not semidefinite. In particular, every such critical point is isolated in the sense that it has neighbourhood $U$ that contains no other critical point. Therefore, a Morse function on a compact manifold $M$ can only have finitely many critical points altogether (see for instance \cite[Corollary 2.19, p. 47]{matsumoto1}). This has the following consequence:

\begin{proposition} \label{propbis:sufficient} Let the set $Y$ be bounded and connected, $m < n$, and assume that the first $m-1$ constraints satisfy the constraint qualifications. In addition, suppose that $g_m|_Y$ is a Morse function. Then, if the last constraint is optimal with respect to the others, the social choice problem over $X$ has a solution.
\end{proposition}
\begin{proof} Every point in $X$ is a critical point of $g_m|_Y$ because of the assumption about the optimality of the last constraint with respect to the others. Since $Y$ is a compact manifold and Morse functions on compact manifolds have only finitely many critical points, it follows that $X$ is finite. This implies that the social choice problem over $X$ indeed has a solution, as discussed above.
\end{proof}

\begin{remark} Suppose that the set $X$ is designed by an agent willing to avoid a social choice paradox. To satisfy the necessary optimality condition the agent would solve the constrained optimization problem \[(P) : \left\{ \begin{array}{lcl} \text{optimize} & & g_m(p) \\ \text{subject to} & &p \in Y; \text{ that is, } \\ & & g_1(p) = c_1, \ldots, g_{m-1}(p) = c_{m-1} \end{array} \right.\] which is an elementary problem in optimization theory which is commonly solved by means of the Lagrangian function \[L(\lambda_1,\ldots,\lambda_{m-1},p) = g_m(p) - \sum_{i=1}^{m-1} \lambda_i \left( g_i(p) - c_i \right).\] It is easy to show that: (i) critical points of $g_m|_Y$ correspond exactly to (unconstrained) critical points of $L$ and (ii) the Hessian of $g_m|_Y$ is nonsingular at a critical point if and only if the same is true of the Hessian of $L$. Thus, the condition laid out in Proposition \ref{propbis:sufficient} that $g_m|_Y$ be a Morse function can be restated in the following familiar and explicitly checkable form: at every critical point of $L$, the Hessian of $L$ has a nonzero determinant.
\end{remark}

Consider the set \[\mathcal{M} := \{g \in \mathcal{C}^{\infty}(\mathbb{R}^n,\mathbb{R}) : g|_Y \text{ is a Morse function}\}.\]

According to Proposition \ref{propbis:sufficient}, whenever $g_m$ belongs to $\mathcal{M}$ the social choice problem over both $X_{g_m}^{\rm max}$ and $X_{g_m}^{\rm min}$ has a solution. Therefore, to establish Theorem \ref{teo:iff} we only need to show that $\mathcal{M}$ is open and dense in $\mathcal{C}^{\infty}(\mathbb{R}^n,\mathbb{R})$. It is well known that Morse functions on a manifold, in this case $Y$, form and open and dense subset of $\mathcal{C}^{\infty}(Y,\mathbb{R})$ (this can be proved by combining \cite[Lemma 2.26, p. 52]{matsumoto1} and \cite[Theorem 2.20, p. 47]{matsumoto1}). However, this does not quite accord to our present situation since we are interested in functions $g$ defined on \emph{all} of $\mathbb{R}^n$ (rather than only on $Y$) and whose \emph{restriction} to $Y$ is a Morse function. Thus there is still some work left to do, and we devote the following subsection to some preparatory results in this direction.

\subsection{Density of Morse functions} Suppose $M$ is a submanifold of some $\mathbb{R}^n$ and $g : M \longrightarrow \mathbb{R}$ is a smooth function. Consider linear perturbations $h$ of $g$ of the form \[M \ni (x_1,\ldots,x_n) \stackrel{h}{\longmapsto} g(x_1,\ldots,x_n) - a_1 x_1 - \ldots - a_nx_n \in \mathbb{R}\] for some coefficients $a_1,\ldots,a_n \in \mathbb{R}$. We are going to prove that these linear perturbations $h$ are almost always Morse functions on $M$. More precisely:

\begin{proposition} \label{prop:linear_pert} Denote by $(x_1,\ldots,x_n)$ the standard Cartesian coordinates of $\mathbb{R}^n$. Let $M \subseteq \mathbb{R}^n$ be a closed smooth manifold and let $g : M \longrightarrow \mathbb{R}$ be a smooth function. Then the set of vectors $(a_1,\ldots,a_n) \in \mathbb{R}^n$ for which the linearly perturbed map \[h(x_1,\ldots,x_n) := g(x_1,\ldots,x_n) - a_1 x_1 - \ldots - a_n x_n\] has some degenerate critical point on $M$ (that is, $h$ is not a Morse function) has Lebesgue measure zero.
\end{proposition}

The proof requires an auxiliary result, which is essentially \cite[Lemma 2.26, p. 52]{matsumoto1}. We state it in a slightly stronger form that follows immediately from its proof (see, in particular, how the coefficients $a_i$ are chosen in \cite[p. 50]{matsumoto1} using Sard's theorem):
\smallskip

{\it Lemma.} Let $M$ be a closed smooth manifold of dimension $d$ and let $U \subseteq M$ be a coordinate chart with coordinates $(y_1,\ldots,y_d)$ [thus a point $q \in U$ has coordinates $(y_1(q),\ldots,y_d(q)) \in \mathbb{R}^d$]. Let $f : U \longrightarrow \mathbb{R}$ be any smooth function. Then, the set of vectors $(a_1, \ldots, a_d) \in \mathbb{R}^d$ for which the map \[U \ni q \longmapsto f(q)  - a_1y_1(q) + \ldots - a_d y_d(q)\] has some degenerate critical point in $U$ has Lebesgue measure zero.
\smallskip

\begin{proof}[Proof of Proposition \ref{prop:linear_pert}] Let $d$ be the dimension of $M$. Pick a point $p \in M$ and consider $T_p M$, the subspace tangent to $M$ at $p$. This is a $d$--dimensional subspace of $\mathbb{R}^n$, so there exist $d$ coordinates among $x_1,\ldots,x_n$ that parameterize the points in $T_p M$, and we assume for notational convenience that these coordinates are $x_1, \ldots, x_d$. The projection map $\pi : (x_1,\ldots,x_n) \longmapsto (x_1,\ldots,x_d)$ then restricts to an isomorphism of $T_p M$ onto $\mathbb{R}^d$. As a consequence of the inverse function theorem, $p$ has an open neighbourhood $U$ in $M$ such that $\pi|_U : U \longrightarrow \pi(U) \subseteq \mathbb{R}^d$ is a diffeomorphism onto its image; that is, $(x_1,\ldots,x_d)$ provides a coordinate system on $U$.

Let $K \subseteq U$ be a compact neighbourhood of $p$ and consider the set \[A_K = \left\{(a_1,\ldots,a_n) \in \mathbb{R}^n : g - \sum_i a_i x_i\text{ has at least one degenerate critical point on $K$}\right \}.\] Clearly $A_K$ is a closed set. We are going to show that it has measure zero. To this end we shall consider the intersections of $A_K$ with the subspaces of $\mathbb{R}^n$ that result from fixing the last $n-d+1$ components of $(a_1,\ldots,a_n)$, and show that each of these intersections has measure zero as a subset of $\mathbb{R}^d$. The fact that $A_K$ is a Borel set (because it is closed) then implies that $A_K$ itself has measure zero. (This is essentially Cavalieri's principle; a particular case of Tonelli's iterated integration theorem: see for instance \cite[Theorem 5.1.4, p. 145]{cohn1}.)

Fix, then, values $a^0_{d+1}, \ldots, a^0_n$ and consider the auxiliary map $f := g - \sum_{i=d+1}^n a_i^0 x_i$. In terms of $f$ we may write the intersection $A_K \cap \{a_{d+1} = a_{d+1}^0, \ldots, a_n = a_n^0\}$, viewed as a subset of $\mathbb{R}^d$, as \begin{equation} \label{eq:A_K} \left\{(a_1, \ldots, a_d) \in \mathbb{R}^d : f - \sum_{i=1}^d a_i x_i \text{ has at least one degenerate critical point on $K$}\right\}.\end{equation}

Although the map $f$ is defined on all of $M$, we are only interested in its behaviour on $U$, where the coordinates $(x_1,\ldots,x_d)$ are valid. Evidently the intersection $A_K \cap \{a_{d+1} = a_{d+1}^0, \ldots, a_n = a_n^0\}$, as described by Equation \eqref{eq:A_K}, is contained in \[\left\{(a_1, \ldots, a_d) \in \mathbb{R}^d : f - \sum_{i=1}^d a_i x_i \text{ has at least one degenerate critical point on $U$}\right\}\] and this set has measure zero according to the lemma stated just before this proposition (applied to the coordinate system $(y_1,\ldots,y_d) = (x_1,\ldots,x_d)$). Consequently its subset $A_K \cap \{a_{d+1} = a_{d+1}^0, \ldots, a_n = a_n^0\}$ has measure zero too. This is true for any $(a_{d+1}^0,\ldots,a_n^0)$ so $A_K$ has measure zero by Cavalieri's principle, as argued earlier.

Now we can easily finish the proof of the proposition. Performing the above construction for every $p \in M$ yields a family of compact neighbourhoods $K_p$ of $p$ such that the corresponding $A_{K_p}$ has measure zero. Since $M$ is compact, one may cover it with only finitely many of the $K_p$; say $K_{p_1},\ldots,K_{p_r}$. Then the set $A := \bigcup_{j=1}^r A_{K_{p_j}}$ is closed, has measure zero, and any $(a_1,\ldots,a_n) \not\in A$ satisfies, by construction, that it does not have any degenerate critical points in any of the $K_p$ and so it does not have any degenerate critical points in $M$, as was to be shown.
\end{proof}

\subsection{The proof of Theorem \ref{teo:iff}} As mentioned earlier, we only need to show that the set \[\mathcal{M} = \{g \in \mathcal{C}^{\infty}(\mathbb{R}^n,\mathbb{R}) : g|_Y \text{ is a Morse function}\}\] is open and dense in $\mathcal{C}^{\infty}(\mathbb{R}^n,\mathbb{R})$.
\smallskip

(1) Openness. Pick a map $g \in \mathcal{M}$, so that $g|_Y$ is a Morse function. This means that $g|_Y$ has only finitely many critical points and the Hessian of $g|_Y$ is nonsingular (computed in some, and hence any, coordinate system) at each of them. Recall that the critical points of $g|_Y$ are precisely those $p \in Y$ where the gradient of $g$ is perpendicular to $Y$. Let $U$ be a neighbourhood of the critical points of $g|_Y$ so small that (condition (i)) the Hessian of $g|_Y$ is nonsingular everywhere in $U$. Notice also that since $g|_Y$ has no critical points outside $U$ by construction the gradient of $g$ is never perpendicular to $Y$ outside $U$ (condition (ii)). Since both (i) and (ii) depend only on the first and second derivatives of $g$, they are clearly satisfied for any sufficiently small perturbation of $g$ in the $\mathcal{C}^2$--topology. Thus $\mathcal{M}$ is an open subset of $\mathcal{C}^{\infty}(\mathbb{R}^n,\mathbb{R})$.
\smallskip

(2) Density. We have to show that given any $g \in \mathcal{C}^{\infty}(\mathbb{R}^n,\mathbb{R})$ and any $\epsilon > 0$ there is $g' \in \mathcal{C}^{\infty}(\mathbb{R}^n,\mathbb{R})$ that is $\epsilon$--close to $g$ and such that $g' \in \mathcal{M}$; that is, $g'|_Y$ is a Morse function. This is a relatively straightforward application of Proposition \ref{prop:linear_pert}, but some care has to be exercised. Let $V$ be an compact neighbourhood of $Y$ in $\mathbb{R}^n$ and pick a smooth function $\theta : \mathbb{R}^n \longrightarrow [0,1]$ such that $\theta|_Y \equiv 1$ and $\theta \equiv 0$ outside $V$. Consider the map \[g'(x_1,\ldots,x_n) := g(x_1,\ldots,x_n) - \theta(x_1,\ldots,x_n) \cdot \sum_i a_i x_i.\] We claim that $g-g'$ and its derivatives up to second order (with respect to the $x_i$) can be made everywhere less than $\epsilon$ just by taking all the $a_i$ sufficiently small. Let us consider, for instance, its second derivatives. A straightforward computation shows that \begin{equation} \label{eq:g_prime} \frac{\partial^2}{\partial x_k \partial x_j}(g-g') = \frac{\partial^2 \theta}{\partial x_k \partial x_j} \cdot \sum_i a_i x_i + \frac{\partial \theta}{\partial x_j} a_k + \frac{\partial \theta}{\partial x_k} a_j.\end{equation} Since $\theta$ is constant outside $V$, all the partial derivatives of $\theta$ vanish there and therefore the same is true of whole right hand side of Equation \eqref{eq:g_prime}. Since $V$ is compact, all the partial derivatives of $\theta$ are be bounded above on $V$ by some constant $S$. Also, there exists a number $R$ such that $V$ is entirely contained in the cube of side $2R$ centered at the origin; that is, $|x_i| \leq R$ whenever $(x_1,\ldots,x_n) \in V$. Then the absolute value of the right hand side of Equation \ref{eq:g_prime} can be bounded above by $SR \sum_i |a_i| + S|a_k| + S|a_j|$. This makes it clear that choosing all the components of $(a_1,\ldots,a_n)$ to satisfy $|a_i| < \delta$ for a suitable $\delta > 0$ one can achieve that $g'$ is $\epsilon$--close to $g$ in the $\mathcal{C}^2$--topology.

Finally, notice that by construction $g' = g - \sum_i a_i x_i$ on $Y$. According to Proposition \ref{prop:linear_pert} the set $A$ of vectors $(a_1,\ldots,a_n) \in \mathbb{R}^n$ for which this the restriction $g'|_Y$ is a Morse function on $Y$ has full measure in $\mathbb{R}^n$. In particular, $A$ has a nonempty intersection with the open cube centered at the origin and with diameter $\delta$, and for any $(a_1,\ldots,a_n)$ in that intersection one has that $g'|_Y$ is a Morse function and $g'$ is $\epsilon$--close to $g$. This finishes the proof.

\section{Appendix A: proofs of Lemmas \ref{lem:var1} and \ref{lem:var2}}

In this appendix we establish Lemmas \ref{lem:var1} and \ref{lem:var2}. As mentioned earlier, we need to use some homology theory. Since this topic is rather elaborate we cannot even recall here the basic definitions, so we refer the interested reader to the book by Hatcher \cite{hatcher} and limit ourselves to state the results that we need.

To any space $U$ we may assign a sequence of (real) vector spaces $H_j(U;\mathbb{R})$ for $j = 0,1,2,\ldots$ which capture some geometric information about $U$. These are called the $j$--dimensional \emph{homology groups} of $U$ (even though they are actually vector spaces) with coefficients in the real numbers. For simplicity we shall just speak of the homology groups of $U$ and denote them $H_j(U)$, supressing $\mathbb{R}$ from the notation.

The following properties hold:

\begin{itemize}
	\item[({\it a}\/)] The dimension of $H_0(U)$ is the number of path connected components of $U$.
  \item[({\it b}\/)] If $U$ is contractible, then $H_0(U) = \mathbb{R}$ and $H_j(U) = \{0\}$ for every $j \geq 1$.
  \item[({\it c}\/)] Poincar\'e duality: if $U$ is a compact, boundariless manifold of dimension $d$, then $H_j(U) = H_{d-j}(U)$ for every $j$.
\end{itemize}

With these properties, the proof of Lemma \ref{lem:var1} reduces to a simple computation:

\begin{proof}[Proof of Lemma \ref{lem:var1}] Let $U$ be a connected component of $M$. Then $U$ is itself a compact manifold of dimension $d \geq 1$ and without boundary. We have $H_0(U;\mathbb{R}) = \mathbb{R}$ by ({\it a}\/) above, because $U$ is connected. By Poincar\'e duality $H_d(U;\mathbb{R}) = H_0(U;\mathbb{R}) = \mathbb{R}$, and it follows from ({\it b}\/) that $U$ is not contractible, because it has a homology group of dimension $\geq 1$ (namely, its $d$--dimensional homology group) which is nonzero.
\end{proof}

The proof of the second lemma requires the use of relative homology groups, which are defined not just for a space $U$ but for a pair $(U,U_0)$ formed by a space $U$ and a subset $U_0$ of $U$. For each $j = 0,1,2, \ldots$ there is a real vector space $H_j(U,U_0)$ called the $j$--dimensional \emph{relative homology group of the pair} $(U,U_0$) with coefficients in the real numbers. There is a relation between the relative homology of a pair $(U,U_0)$ and the homology groups of both $U$ and $U_0$ which is expressed by a so-called long exact sequence as follows. For each dimension $j$ there are linear maps $H_j(U_0) \longrightarrow H_j(U)$, $H_j(U) \longrightarrow H_j(U,U_0)$ and $H_j(U,U_0) \longrightarrow H_{j-1}(U_0)$ that fit into a sequence \[\ldots \longrightarrow H_2(U) \longrightarrow H_2(U,U_0) \longrightarrow H_1(U_0) \longrightarrow H_1(U) \longrightarrow H_1(U,U_0) \longrightarrow \{0\}\] (which continues to the left in the same fashion) having the property of being \emph{exact}: the image of the map entering any one of the terms of the sequence coincides with the kernel of the map connecting that term to the one to its right.

In addition to this, we shall also make use of
\begin{itemize}
	\item[({\it d}\/)] Lefschetz duality: if $U$ is a compact manifold (with boundary) of dimension $d$, then $H_j(U,\partial U) = H_{d-j}(U)$ for every $j$.
\end{itemize}

\begin{proof}[Proof of Lemma \ref{lem:var2}] Since $M$ is contractible and $d-1 \geq 1$ by assumption, property ({\it b}\/) entails that $H_{d-1}(M) = \{0\}$. By Lefschetz duality (applied to $U = M$) one then has $H_1(M,\partial M) = H_{d-1}(M) = \{0\}$. Again from ({\it b}\/) one also has $H_0(M) = \mathbb{R}$.

Consider the following portion of the long exact sequence for the pair $(M,\partial M)$: \[H_1(M,\partial M) \stackrel{\alpha}{\longrightarrow} H_0(\partial M) \stackrel{\beta}{\longrightarrow} H_0(M) \stackrel{\gamma}{\longrightarrow} \{0\}\] which, taking into account the previous paragraph, reads \[\{0\} \stackrel{\alpha}{\longrightarrow} H_0(\partial M) \stackrel{\beta}{\longrightarrow} \mathbb{R} \stackrel{\gamma}{\longrightarrow} \{0\}.\] Clearly both $\alpha$ and $\gamma$ must be the zero homomorphisms, so the image of $\alpha$ is zero and the kernel of $\gamma$ is all of its source space $\mathbb{R}$. Thus by the exactness of the sequence, the kernel of $\beta$ is zero too, so that $\beta$ is injective, and the image of $\beta$ is $\mathbb{R}$. Hence $\beta$ establishes an isomorphism between $H_0(\partial M)$ and $\mathbb{R}$. Using ({\it a}\/) we conclude that $\partial M$ is connected.
\end{proof}

\section{Appendix B: proofs of Lemmas \ref{lem:extend1} and \ref{lem:morse}} \label{app:b}

\subsection{Proof of Lemma \ref{lem:extend1}} It is both notationally and conceptually simpler to prove a slightly more general result, from which Lemma \ref{lem:extend1} follows letting $M = X_{[u_1,u_2]}$ and $Z = X_u$:

\begin{lemma} Let $M$ be a compact manifold and $Z \subseteq M$ a closed subset of $M$. Suppose $F : Z^k \longrightarrow Z$ is an SCF over $Z$. Then there exist a neighbourhood $U$ of $Z$ in $M$ and a continuous map $F_U : U^k \longrightarrow M$ such that $F_U$ is unanimous and anonymous.
\end{lemma}
\begin{proof} Think of $F$ as a mapping $F : Z^k \longrightarrow M$ and extend it setting $F(p,\ldots,p) = p$ for every $p \in M$. Now its domain is \[D := Z^k \cup \{(p,\ldots,p) : p \in M\},\] which is a compact subset of $M^k$. Clearly $F$ is still continuous on this new larger domain $D$.

Consider the quotient space obtained from $M^k$ by identifying, via an equivalence relation $\sim$, each $k$--tuple $(p_1,\ldots,p_k)$ with all of its permutations. We shall denote $\pi : M^k \longrightarrow M^k/\sim$ the canonical projection. The set $D$ projects onto a compact subset $\pi(D)$ of $M^k/\sim$. In turn the map $F$, due to its invariance under permutation of its arguments, descends to a continuous map \[\bar{F} : \pi(D) \longrightarrow M.\]

Now we make use of the following extension result: every continuous map from a closed subset of a metric space into a manifold $M$ can be extended to a continuous map defined on a neighbourhood $W$ of the subset (see for instance \cite[Proposition 8.3, p. 47]{hu1}). Applying this result to the closed subset $\pi(D)$ of the metric space $M^k/\sim$ and the map $\bar{F}$ we see that the latter can be extended continuously to a neighbourhood $W$ of $\pi(D)$ in $M^k/\sim$. For notational ease the extension will still be denoted $\bar{F}$.

The set $\pi^{-1}(W)$ is a neighbourhood of $D$ in $M^k$, so in particular it is a neighbourhood of $Z^k$. It is easy to see that there exists a neighbourhood $U$ of $Z$ in $M$ such that $U^k \subseteq \pi^{-1}(W)$. Then the map \[F_U : U^k \longrightarrow M \ \ ; \ \ F_U(p_1,\ldots,p_k) = (\bar{F} \circ \pi)(p_1,\ldots,p_k)\] provides the desired extension: it is clearly continuous and unanimous, and it is also anonymous because any two permutations of a $k$--tuple $(p_1,\ldots,p_k)$ are projected by $\pi$ onto the same element of $M^k/\sim$.
\end{proof}

\subsection{Proof of Lemma \ref{lem:morse}} The construction of the map $r$ is rather indirect: we shall define a tangent vectorfield on $Y$, consider the flow $\varphi$ that it generates and then use $\varphi$ to define $r$. This approach is closely related to Morse theory, and a quick glance at the book by Milnor \cite[pp. 12 and 13]{milnor1} may be useful. Some acquaintance with differential geometry is required to follow the argument.

Recall that $Y \subseteq \mathbb{R}^n$ is a differentiable manifold defined by the constraints $g_i(p) = c_i$ for $i = 1, 2, \ldots, m-1$. At any point $p \in Y$ their gradients are all orthogonal to $Y$ or, otherwise stated, the tangent space to $Y$ at $p$ is the subspace of $\mathbb{R}^n$ orthogonal to all the $\{\nabla g_i(p) : 1 \leq i \leq m-1\}$. Denote $V(p)$ the projection of $\nabla g_m(p)$ onto that tangent space, thus obtaining a tangent vectorfield $p \longmapsto V(p)$ on $Y$. This vectorfield $V(p)$ can be given a very rough but rather helpful intuitive interpretation: inasmuch as $\nabla g_m(p)$ tells us the direction along which $f$ increases most quickly, its projection $V(p)$ tells us in what direction we should advance to obtain the quickest increase of $g_m$ \emph{while remaining in $Y$}.
\smallskip

{\it Assertion 1.} $V(p)$ is zero precisely when $p$ is a critical point of $g_m|_Y$.
\begin{proof} Notice that $V(p)$ is zero precisely when $\nabla g_m(p)$ is orthogonal to $Y$ at $p$; that is to say, precisely when $\nabla g_m(p)$ is a linear combination of the gradients $\{\nabla g_i(p) : 1 \leq i \leq m-1\}$ or, equivalently, when $p$ is a critical point of $g_m|_X$.
\end{proof}

{\it Assertion 2.} The scalar product $\nabla g_m(p) \cdot V(p)$ is always nonnegative and it is actually positive when $p$ is not a critical point of $g_m|_Y$.
\begin{proof} Observe that by construction the angle between $\nabla g_m(p)$ and $V(p)$ is at most ninety degrees, so the scalar product $\nabla g_m(p) \cdot V(p)$ is always nonnegative. Together with Assertion 1, this proves the result.
\end{proof}

Using $V(p)$ we define a new tangent vectorfield $W : p \longmapsto (u-g_m(p)) V(p)$. Let $\varphi : Y \times \mathbb{R} \longrightarrow Y$ be the flow generated on $Y$ by the vectorfield $W$. Notice that, since $Y$ is compact, $\varphi$ is globally defined.

To gain some geometric intuition consider again the situation shown in Figure \ref{fig:explanation}. Recall that $g_2(x,y,z) = z$, so its gradient is the constant vector $(0,0,1)$. The vectorfield $V(p)$ points, at each $p \in Y$, in the direction that yields the quickest increase in height while remaining within $Y$. The vector field $W(p)$ is obtained multiplying $V(p)$ by the modulating factor $u - g_2(p) = u - z$, which is zero precisely on $X_u$, negative above $X_u$ and positive below $X_u$. Taking into account these signs, $W(p)$ is zero on $X_u$, points in the direction of quickest \emph{descent} if $p$ is above $X_u$, and points in the direction of quickest ascent if $p$ is below $X_u$. The vectors $W(p)$ also become shorter as $p$ approaches $X_u$, since the factor $u - g_2(p)$ becomes closer to zero. Figure \ref{fig:lema9}provides a sketch of $W(p)$ (shown as arrows tangent to $X_{[u_1,u_2]}$) and of some of the trajectories of the flow $\varphi$ obtained by integrating $W(p)$.

If we follow the directions of $W(p)$, starting at any point $p$, it seems intuitively clear that we will move towards $X_u$ advancing ever more slowly, since $W(p)$ (which is our speed) becomes smaller the closer we get to $X_u$. Unless $p \in X_u$, in which case we would actually stay still since $W(p) = 0$, we would approach $X_u$ asymptotically but never get there. In any case, there will be a finite time $t_p$ at which we will enter any prescribed neighbourhood $U$ of $X_u$ and never leave it again. The map $r$ that we are looking for will essentially be defined as $r(p) = \text{the point we reach at time } t_p$. In our trip from $p$ to $r(p)$ we might follow a simple path like the ones shown to the left side of the drawing or, possibly, a much more complicated one which approaches $X_u$ spiralling around it or in some other strange fashion.

\begin{figure}[h!]
\begin{pspicture}(0,0)(8,3.5)
\rput[bl](0.5,0){\scalebox{0.5}{\includegraphics{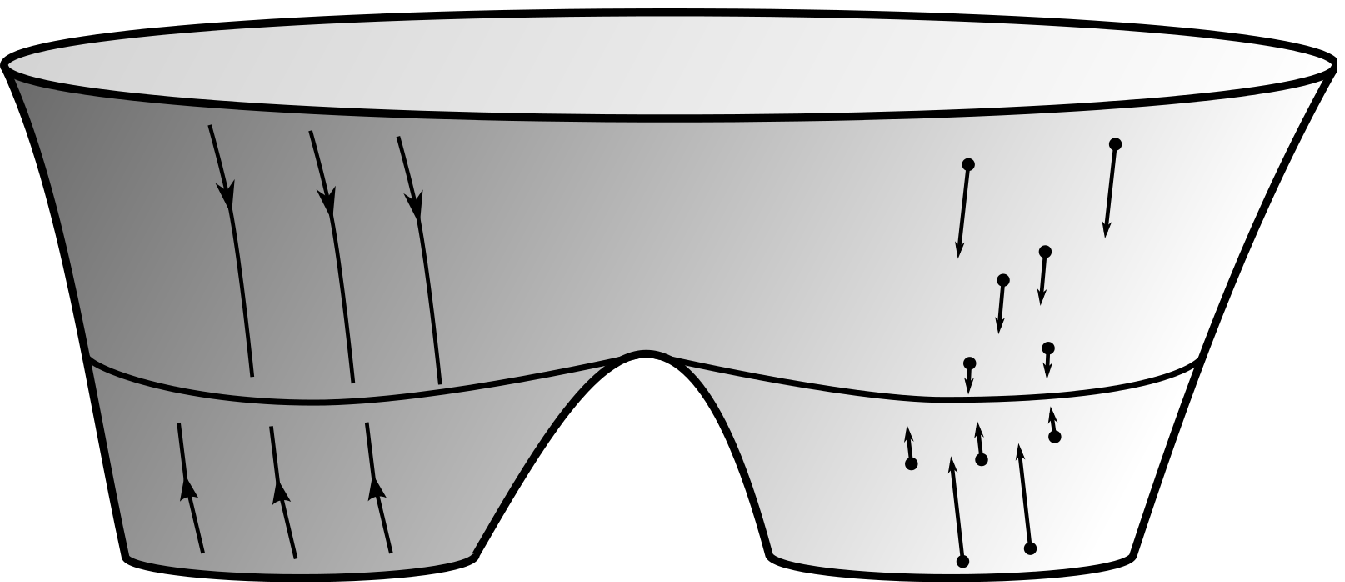}}}
\rput(2,3.2){$X_{[u_1,u_2]}$}
\end{pspicture}
\caption{\label{fig:lema9}}
\end{figure}

Let us go back to mathematics again. The following proposition collects some properties of $\varphi$ that are the formal counterparts of the ideas just described:

\begin{proposition} \label{prop:flujo} The flow $\varphi$ has the following properties:
\begin{itemize}
        \item[(1)] For every $q \in X_{[u_1,u_2]}$ and every $t \geq 0$, the point $\varphi(q,t)$ belongs to $X_{[u_1,u_2]}$ too.
        \item[(2)] For every neighbourhood $U$ of $X_u$ in $X_{[u_1,u_2]}$ there exists $T > 0$ such that $\varphi(q,t) \in U$ for every $q \in X_{[u_1,u_2]}$ and every $t \geq T$.
\end{itemize}
\end{proposition}

In the parlance of dynamical systems, (1) means that $X_{[u_1,u_2]}$ is positively invariant under $\varphi$ and (2) states that $X_u$ is a stable attractor in $X_{[u_1,u_2]}$. As a preparation to prove the proposition we are going to investigate some qualitative properties of the trajectories of $\varphi$.

Fix a point $q \in X_{[u_1,u_2]}$ and let $\gamma$ be the trajectory of $\varphi$ with initial condition $\gamma(0) = q$ (in terms of the flow, $\gamma(t) = \varphi(q,t)$). More explicitly, $\gamma: \mathbb{R} \longrightarrow Y$ is a smooth curve in $Y$ which satisfies $\gamma(0) = q$ and \[\frac{{\rm d} \gamma}{{\rm d}t}(t) = W(\gamma(t))\] for every $t \in \mathbb{R}$ (that is, $\gamma$ is an integral curve of the vectorfield $W$). We are interested in the behaviour of $\gamma(t)$ for $t \geq 0$, and for definiteness we consider the case $u < f(q) \leq u_2$.
\smallskip

{\it Assertion 3.} The inequality $u < g_m(\gamma(t))$ holds for every $t \in \mathbb{R}$.
\begin{proof} Each point of $X_u$ is a zero of $W$ and therefore a fixed point of the flow $\varphi$. Since the trajectory $\gamma$ goes through the point $q$, which does not belong to $X_u$, it follows that $\gamma(t) \not\in X_u$ for every $t \in \mathbb{R}$. Consider the map $t \longmapsto (g_m \circ \gamma)(t)$. By what we have just seen, it never attains the value $u$, so (as it continuous) it must be the case that $g_m(\gamma(t))$ is either always $> u$ or $< u$. Since we have taken $g_m(q) > u$, it follows that $g_m(\gamma(t)) \neq u$ for every $t \in \mathbb{R}$.
\end{proof}

{\it Assertion 4.} The inequality $g_m(\gamma(t)) \leq u_2$ holds for every $t \geq 0$.
\begin{proof} Using the chain rule we compute the time derivative of the map $t \longmapsto (g_m \circ \gamma)(t)$ as follows: \[ \frac{{\rm d}}{{\rm d}t}(g_m \circ \gamma)(t) = \nabla g_m(\gamma(t)) \cdot \frac{{\rm d}\gamma}{{\rm d}t}(t) = \nabla g_m(\gamma(t)) \cdot W(\gamma(t)),\] and since by definition $W(p) = (u-g_m(p)) V(p)$, we have $\nabla g_m(p) \cdot W(p) = (u-g_m(p)) \nabla g_m(p) \cdot V(p)$, so \begin{equation} \label{eq:der} \frac{{\rm d}}{{\rm d}t} (g_m \circ \gamma)(t) = (u-g_m(\gamma(t))) \nabla g_m(\gamma(t)) \cdot V(\gamma(t)).\end{equation}

The right hand side is the product of two factors. The first is $u - f_m(\gamma(t))$, which is strictly negative by the previous assertion, and the second is $\nabla g_m(\gamma(t)) \cdot V(\gamma(t)) \geq 0$ which is nonnegative by Assertion 2. Thus the derivative of $t \longmapsto (g_m \circ \gamma)(t)$ is nonpositive, and so the map is nonincreasing. In particular, since at $t = 0$ we have $(g_m \circ \gamma)(t) = g_m(q) \leq u_2$, this same inequality holds for all $t \geq 0$.
\end{proof}

{\it Assertion 5.} For any $p \in X_{[u_1,u_2]}$ such that $u < g_m(p) \leq u_2$ the inequality \[(u-g_m(p)) \nabla f(p) \cdot V(p) < 0\] holds true.

\begin{proof} Since $u - g_m(p) < 0$, we only need to prove that $\nabla g_m(p) \cdot V(p) > 0$. By Assertion 2, this scalar product is always nonnegative and it is zero precisely when $p$ is a critical point of $g_m|_X$. Now, the choice of $u_1$ and $u_2$ guarantees that the critical points that $g_m|_Y$ may have in the set $X_{[u_1,u_2]}$ are all contained in $X_u$. Since $p \not\in X_u$, the assertion follows.
\end{proof}

{\it Assertion 6.} $g_m(\gamma(t)) \rightarrow u$ as $t \rightarrow +\infty$.
\begin{proof} Notice that $t \longmapsto (g_m \circ \gamma)(t)$ must indeed converge to some $u_*$ as $t \rightarrow +\infty$ because according to the computation in Assertion 4 it is a monotonous nonincreasing function bounded below by $u$. Let us assume that $u_*$ is strictly larger than $u$ and arrive at a contradiction.

Set $D(p) = (u-g_m(p)) \nabla g_m(p) \cdot V(p)$ for brevity. Let $X_{u_*} = \{p \in Y : g_m(p) = u_*\}$. This set is closed in $Y$, so it is compact. Also, the previous assertion says that $D(p) < 0$ for every $q \in X_{u_*}$. Since $X_{u_*}$ is compact and $D$ is continuous, there is an $\epsilon < 0$ such that $D(p) < \epsilon$ for every $p \in X_{u_*}$. In fact, more is true: there is a neighbourhood $U$ of $X_{u_*}$ in $X_{[u_1,u_2]}$ where the same inequality holds; that is, $D(p) < \epsilon$ for every $p \in U$.

We are almost finished. Since $g_m(\gamma(t)) \rightarrow u_*$ as $t \rightarrow +\infty$, there exists $T > 0$ such that $\gamma(t) \in U$ for every $t > T$. By the mean value theorem, for any $t$ there exists $\xi_t$ between $t$ and $t+1$ such that \begin{equation} \label{eq:mvt} g_m(\gamma(t+1)) - g_m(\gamma(t)) = \frac{{\rm d}}{{\rm d}t} (g_m \circ \gamma)(\xi_t) = D(\gamma(\xi_t)),\end{equation} where in the last equality we have used equation \eqref{eq:der}. Let us consider what happens in the above expression when $t \rightarrow +\infty$. Since $\xi_t$ lies between $t$ and $t+1$, as soon as $t > T$ we also have $\xi_t > T$ and therefore $\gamma(\xi_t) \in U$, which entails $D(\gamma(\xi_t)) < \epsilon$. Thus the right hand side of \eqref{eq:mvt} is bounded away from $0$ (recall that $\epsilon < 0$). However, its left hand side converges to $0$ as $t \rightarrow +\infty$ because both summands converge to $u_*$. This contradiction finishes the proof.
\end{proof}

\begin{proof}[Proof of Proposition \ref{prop:flujo}] (1) We have seen that for an initial condition $q = \gamma(0)$ satisfying $u < g_m(q) \leq u_2$, the trajectory $\gamma(t)$ remains in the set $\{p \in Y : u < g_m(p) \leq u_2\}$ for every $t \geq 0$. Evidently, if the initial condition $q$ satisfies $u_1 \leq g_m(q) < u$, similar arguments show that $\gamma(t)$ remains in the set $\{p \in Y : u_1 \leq g_m(p) < u\}$ for all $t \geq 0$. The remaining case, $g_m(q) = u$, is very simple: $q$ is then a zero of the vectorfield $W$ and so $\gamma(t) = q$ for every $t \in \mathbb{R}$. Summing up, for an initial condition $q \in X_{[u_1,u_2]}$ the trajectory $\gamma$ remains in the set $X_{[u_1,u_2]}$. Part (1) of the proposition is just a restatement of this, since in terms of the flow the trajectory $\gamma$ with initial condition $q = \gamma(0)$ is simply $\gamma(t) = \varphi(q,t)$.

(2) Find $u'_1$ and $u'_2$ such that $u_1 < u'_1 < u < u'_2 < u_2$ and $X_{[u'_1,u'_2]} \subseteq U$. In accordance with the notation we have been using so far, denote \[X_{(u'_1,u'_2)} = \{p \in Y : u'_1 < g_m(p) < u'_2\},\] which is an open subset of $X_{[u_1,u_2]}$ by continuity of $g_m$. In fact it is a neighbourhood of $X_u$ in $X_{[u_1,u_2]}$, so by Assertion 6 for each $q \in X_{[u_1,u_2]}$ there exists $t_q \geq 0$ such that $\varphi(q,t_q) \in X_{(u'_1,u'_2)}$. Now the continuity of $\varphi$ guarantees that $q$ has an open neighbourhood $U_q$ in $X_{[u_1,u_2]}$ such that $\varphi(U_q \times \{t_q\}) \subseteq X_{(u'_1,u'_2)}$. In particular $\varphi(U_q \times \{t_q\}) \subseteq X_{[u'_1,u'_2]}$, and by part (1) of this proposition (applied to $X_{[u'_1,u'_2]}$ rather than $X_{[u_1,u_2]}$) we see that $\varphi(U_q \times \{t\}) \subseteq X_{[u'_1,u'_2]}$ for every $t \geq t_q$. The $U_q$ cover the compact set $X_{[u_1,u_2]}$, so a finite family of them cover it too, say $U_{q_1},U_{q_2},\ldots,U_{q_r}$. Let $T$ be the maximum of $t_{q_1},t_{q_2},\ldots,t_{q_r}$. Then whenever $t \geq T$ we have that $\varphi(q,t) \in X_{[u'_1,u'_2]}$ for every $t \in X_{[u_1,u_2]}$, proving the proposition.
\end{proof}

We are finally ready to prove Lemma \ref{lem:morse}. For the convenience of the reader, we restate it here:

\begin{lem} Given any neighbourhood $U$ of $X_u$ in $X_{[u_1,u_2]}$ there exists a continuous mapping $r : X_{[u_1,u_2]} \longrightarrow X_{[u_1,u_2]}$ such that:
\begin{enumerate}
        \item[(1)] $r(p) \in U$ for every $p \in X_{[u_1,u_2]}$,
        \item[(2)] $r$ is homotopic to the identity in $X_{[u_1,u_2]}$.
\end{enumerate}
\end{lem}
\begin{proof} According to Proposition \ref{prop:flujo} there exists $T \geq 0$ such that $\varphi(p,t) \in U$ for every $p \in X_{[u_1,u_2]}$ and every $t \geq T$. Let $r$ be defined by $r(p) := \varphi(p,T)$. By construction $r(p) \in U$, so indeed satisfies condition (1). Also, $r$ is homotopic to the identity: the flow $\varphi(p,t)$, for $0 \leq t \leq T$, provides a suitable homotopy. Thus the lemma is proved.
\end{proof}

\section{Concluding remarks}

In this paper we have considered the (topological) social choice problem over bounded sets of alternatives $X \subseteq \mathbb{R}^n$ that are defined by equality constraints $g_i(p) = c_i$. We have shown that the problem has a solution in the affirmative if and only if the values $c_i$ of the constraints that define the set $X$  are finely tuned in the following sense: each of them must be either the global maximum or the global minimum of the corresponding $g_i$ over the set defined by the remaining constraints $g_j(p) = c_j$, $j \neq i$. When this condition is satisfied we say that $X$ is optimally designed. This criterion is simple to state and check explicitly and it is elementary in its language, which we feel is a valuable feature that distinguishes it from the results that already exist in the literature.

It should be observed that the need for optimization emerges exclusively from the social choice problem itself and not as a consequence of any particular interpretation of the variables, the constraints, or any assumption concerning rationality, utility functions, or any other element related to Economics. In the particular case when the set of alternatives $X$ does have an economical interpretation, the same is true of the principle of optimal design. This was illustrated in the Introduction with a toy example where the condition of optimal design meant that designer of the problem, if any, should behave rationally in the sense of Economics. We emphasize again that the rationality of the agents involved in the actual choice plays absolutely no role in this reasoning.

There exists an extensive literature on how to avoid social choice paradoxes, for instance by constraining the way in which the agents choose their individual alternative (\cite{Arrow1951}, \cite{black1}, \cite{inada1}), by weakening the axioms that an aggregation rule needs to satisfy (\cite{sen1}), or by requiring that the agents provide richer information as their input for the aggregation rule (\cite{deschampsgevers1}, \cite{sen3}). The interested reader can find many more references in \cite{handbook}. We see that in our present framework, requiring that the designer of the set of alternatives behaves in the way that is almost axiomatic in Economics; that is, rationally, is a necessary and sufficient condition to avoid a social choice paradox.

\section*{Acknowledgments}

The authors warmly thank Carlos Herv\'es, Francisco Marhuenda, Juan Pablo Rinc\'on, and Mar\'ia del Carmen Rold\'an for their critical reading of the paper.

Juan A. Crespo is supported by the Ministerio de Econom\'ia y Competitividad (Spain) grant ECO2013-42710-P and the Comunidad de Madrid grant MadEco-CM (S2015/HUM-3444).

J.J. S\'anchez-Gabites is supported by the Ministerio de Econom\'ia y Competitividad (Spain) grant MTM2015-63612-P and the Comunidad de Madrid grant MadEco-CM (S2015/HUM-3444).

\end{document}